\documentclass[a4paper,10pt]{article}
\usepackage{indentfirst}
\usepackage{amsmath}
\usepackage{authblk}
\usepackage{geometry}
\usepackage{accents}
\usepackage{statex}
\usepackage[francais,english]{babel}
\usepackage{url}

\usepackage[utf8]{inputenc}
\usepackage[nottoc]{tocbibind}

\newtheorem{theorem}{Theorem}
\newtheorem{proposition}{Proposition}

\newenvironment{proof}[1][Proof]{\begin{trivlist}
\item[\hskip \labelsep {\bfseries #1}]}{\end{trivlist}}
\newenvironment{definition}[1][Definition]{\begin{trivlist}
\item[\hskip \labelsep {\bfseries #1}]}{\end{trivlist}}

\title{ On completely multiplicative automatic sequences}
\author{Shuo LI\\
\normalsize{shuo.li@imj-prg.fr}\\}
\date {}

\begin{document}

\maketitle

\begin{abstract}
In this article, we prove that all completely multiplicative automatic sequences $(a_n)_{n \in \mathbf{N}}$ defined on $\mathbf{C}$, vanishing or not, can be written in the form $a_n=b_n\chi_n$, where $(b_n)_{n \in \mathbf{N}}$ is an almost constant sequence, and $(\chi_n)_{n \in \mathbf{N}}$ is a Dirichlet character.
\end{abstract}

\section{Introduction}

In this article, we describe the decomposition of completely multiplicative automatic sequences, which will be referred to as CMAS. In article \cite{Puchta2011}, the author proves that a non-vanishing CMAS is almost periodic (defined in \cite{Puchta2011}). In article \cite{ALLOUCHE2018}, the authors give a formal expression of all non-vanishing CMAS and also some examples in the vanishing case (named mock characters). In article \cite{HU2017}, the author studies completely multiplicative sequences, which will be referred to as CMS, taking values on a general field, that have finitely many prime numbers $p$ such that $a_p \neq 1$; she proves that such CMS have complexity $p_a(n)=O(n^k)$, where $k=\#\left\{p| p\in \mathbf{P}, a_p \neq 1,0\right\}$. In this article, we prove that all completely multiplicative sequences $(a_n)_{n \in \mathbf{N}}$ defined on $\mathbf{C}$, vanishing or not, can be written in the form $a_n=b_n\chi_n$, where $(b_n)_{n \in \mathbf{N}}$ is an almost constant sequence, and $(\chi_n)_{n \in \mathbf{N}}$ is a Dirichlet character.

Let us consider a CMAS $(a_n)_{n \in \mathbf{N}}$ defined on $\mathbf{C}$. We first prove that all CMAS are mock characters (defined in \cite{ALLOUCHE2018}) with an exceptional case. Second, we study the CMAS satisfying the condition $C$~: $$\sum_{p| a_p \neq 1, p \in \mathbf{P}}\frac{1}{p}<\infty,$$ where $\mathbf{P}$ is the set of prime numbers. We prove that in this case, there is at most one prime $p$ such that $a_p \neq 1$ or $0$. In the third part, we prove that all CMAS are either Dirichlet-like sequences or strongly aperiodic sequences. Finally, we conclude by proving that a strongly aperiodic sequence cannot be automatic.

\section{Definitions, notation and basic propositions}

Let us recall the definition of automatic sequences and complete multiplicativity:

\begin{definition}
Let $(a_n)_{n \in \mathbf{N}}$ be an infinite sequence and $k \geq 2$ be an integer; we say that this sequence is $k$-automatic if there is a finite set of sequences containing $(a_n)_{n \in \mathbf{N}}$ and closed under the maps
$$a_n \rightarrow a_{kn+i}, i=0,1,...k-1.$$
\end{definition}

There is another definition of a $k$-automatic sequence $(a_n)_{n \in \mathbf{N}}$ via an automaton. An automaton is an oriented graph with one state distinguished as the initial state, and, for each state, there are exactly $k$ edges pointing from this state to other states; these edges are labeled as $0,1,...,k-1$. There is an output function $f$, which maps the set of states to a set $U$. For an arbitrary $n \in \mathbf{N}$, the $n$-th element of the automatic sequence can be computed as follows: writing the $k$-ary expansion of $n$, start from the initial state and move from one state to another by taking the edge read in the $k$-ary expansion one by one until stopping on some state. The value of  $a_n$ is the evaluation of $f$ on the stopping state. If we read the expansion from right to left, then we call this automaton a reverse automaton of the sequence; otherwise, it is called a direct automaton. \\

In this article, all automata considered are direct automata. \\

\begin{definition}
We define a subword\footnote{ What we call a {\it subword}
     here is also called a {\it factor} in the literature; however, we use {\it
factor} with a different meaning. }
of a sequence as a finite length string of the sequence. We let $\overline{w}_l$ denote a subword of length $l$.\\
\end{definition}

\begin{definition}
Let $(a_n)_{n \in \mathbf{N}}$ be an infinite sequence. We say that this sequence is completely multiplicative if, for any $p, q \in \mathbf{N}$, we have $a_pa_q=a_{pq}.$
\end{definition}

It is easy to see that a CMAS can only take finitely many values, either $0$ or a $k$-th root of unity (see, for example, Lemma 1 \cite{Puchta2011}). \\
 
\begin{definition}
Let $(a_n)_{n \in \mathbf{N}}$ be a CMS. We say that $a_p$ is a prime factor of $(a_n)_{n \in \mathbf{N}}$ if $p$ is a prime number and $a_p \neq 1$. Moreover, we say that $a_p$ is a non-trivial factor if $a_p \neq 0$, and we say that $a_p$ is a $0$-factor if $a_p=0$. We say that a sequence $(a_n)_{n \in \mathbf{N}}$ is generated by $a_{p_1}, a_{p_2},...$ if and only if $a_{p_1}, a_{p_2},...$ are the only prime factors of the sequence.\\
\end{definition}

\begin{definition}
We say that a sequence is an almost-$0$ sequence if there is only one non-trivial factor $a_p$ and $a_q=0$ for all primes $q \neq p$.\\
\end{definition}

\begin{proposition}
Let $(a_n)_{n \in \mathbf{N}}$ be a $k$-CMAS and $q$ be the number of states of a direct automaton generating $(a_n)_{n \in \mathbf{N}}$; then, for any $m, y\in \mathbf{N}$, we have equality between the sets $\left\{a_n|mk^{q!} \leq n < (m+1)k^{q!}\right\}=\left\{a_n|mk^{yq!} \leq n < (m+1)k^{yq!}\right\}$.
\end{proposition}

\begin{proof}
In article \cite{Puchta2011} (Lemma 3 and Theorem 1), the author proves that, in an automaton, every state that can be reached from a specific state, say, $s$, with $q!$ steps, can be reached with $yq!$ steps for every $y \geq1$; conversely, if a state can be reached with $yq!$ steps for some $y \geq 1$, then it can already be reached with $q!$ steps. This proves the proposition. \\
\end{proof}

Let us consider a CMS $(a_n)_{n \in \mathbf{N}}$  taking values in a finite Abelian group $G$. We define
$$E=\left\{g| g \in G, \sum_{ a_p=g,p \in \mathbf{P}}\frac{1}{p} = \infty\right\}$$ and $G_1$ as the subgroup of $G$ generated by $E$.

\begin{definition}
We say that an element $\zeta$ of a sequence $(a_n)_{n \in \mathbf{N}}$ has a natural density if and only if $\lim_{N \to \infty}\frac{\sharp\left\{n| a_n=\zeta, 0 \leq n \leq N\right\}}{N+1}$ exists, and we say that the sequence $(a_n)_{n \in \mathbf{N}}$ has a mean value if and only if $\lim_{N \to \infty}\frac{\sum_{n=0}^{N} a_n}{N+1}$ exists.\\
\end{definition}

\begin{proposition}
Let $(a_n)_{n \in \mathbf{N}}$ be a CMS taking values in a finite Abelian group $G$; then, for all elements $g \in G$, the sequence $a^{-1}(g)=\left\{n: a_n=g\right\}$ has a non-zero natural density. Furthermore, this density depends only on the coset $rG_1$ in which the element $g$ lies. The statement is still true in the case that $G$ is a semi-group generated by a finite group and $0$ under the condition that there are finitely many primes $p$ such that $a_p=0$.
\end{proposition}

\begin{proof}
When $G$ is an Abelian group, the proposition is proved in Theorem 3.10, \cite{Ruzsa1977}, and when $G$ is a semi-group, Theorem 7.3, \cite{Ruzsa1977} shows that all elements in $G$ have a natural density. To conclude the proof, it is sufficient to consider the following fact: let $f_0$ be a CMS such that there exists a prime $p$ with $a_p=0$, and let $f_1$ be another CMS such that
$$f_1(q)=\begin{cases}
f_0(q) \;\text{if}\; q\in \mathbf{P}, q \neq p\\
1 \; \text{otherwise},
\end{cases}$$
If $d_0(g)$ and $d_1(g)$ denote the natural density of $g$ in the sequence $(f_0(n))_{n \in \mathbf{n}}$ and $(f_1(n))_{n \in \mathbf{N}}$, respectively, then we have the equality $$d_1(g)=d_0(g)+\frac{1}{p}d_0(g)+\frac{1}{p^2}d_0(g)...=\frac{p}{p-1}d_0(g).$$ Doing this repeatedly until we obtain a non-vanishing sequence, we can conclude the proof by the first part of the proposition.
\end{proof}

\section{Finiteness of the numbers of $0$-factors}

In this section, we will prove that a CMAS is either a mock character, which means that it has only finitely many $0$-factors, or an almost-$0$ sequence, that is, $a_m=0$ for all $m$ that are not a power of $p$, and $a_{p^k}=\delta^k$ for some $\delta$, where $\delta$ is a root of unity or $0$ and $p$ is a prime number.\\

\begin{proposition}
Let $(a_n)_{n \in \mathbf{N}}$ be a $p$-CMAS; then, it is either a mock character or an almost-$0$ sequence.
\end{proposition}

\begin{proof}
If $(a_n)_{n \in \mathbf{N}}$ is not a mock character, then it contains infinitely many $0$-factors. Here, we prove that, in this case, if there is some $a_m \neq 0$, then $m$ must be a power of $p$, and $p$ must be a prime number. Let us suppose that there are $q$ states of the automaton generating the sequence. As there are infinitely many $0$-factors, it is easy to find a subword of length $p^{2q!}$ such that all its elements are $0$:\\
This is equivalent to finding some $m \in \mathbf{N}$ and $p^{2q!}$ $0$-factors, say, $a_{p_1},a_{p_2},...,a_{p_{p^{2q!}}}$, such that
$$
 \left\{
    \begin{array}{ll}
	m \equiv 0 \pmod {p_1}\\
	m+1 \equiv 0 \pmod {p_2}\\
	m+2 \equiv 0 \pmod {p_3}\\
	...\\
	m+p^{2q!}-1 \equiv 0 \pmod {p_{p^{2q!}}}
    \end{array}
\right.
$$
If $m$ is a solution of the above system, then the subword $\overline{a_ma_{m+1}...a_{m+p^{2q!}-1}}$ is all $0$'s. Therefore, there exists an $m'$ such that $m \leq m'p^{q!}<  (m'+1)p^{q!} \leq m+p^{2q!}$. Because of Proposition 1, for any $y \in \mathbf{N}$, $a_{k}=0$ for all $k$ such that $m'p^{yq!} \leq k < (m'+1)p^{yq!}$. Taking an arbitrary prime $r$, if $r$ and $p$ are not multiplicatively dependent, then $a_r=0$ because there exists a power of $r$ satisfying $m'p^{yq!} \leq r^t < (m'+1)p^{yq!}$. This inequality holds because we can find some integer $t$ and $y$ such that
$$\log_p m' \leq t\log_p r-yq!<\log_p (m'+1).$$

The above argument shows that if $(a_n)_{n \in \mathbf{N}}$ is not a sequence such that $a_m=0$ for all $m>1$, then $p$ must be a power of a prime number $p'$. Otherwise, as $p$ is not multiplicatively dependent with any prime numbers, $a_m=0$ for all $m>1$. Furthermore, the sequence $(a_n)_{n \in \mathbf{N}}$ can have at most one non-zero prime factor, and if it exists, it should be $a_{p'}$. Using automaticity, we can replace $p'$ with $p$.
\end{proof}

\section{CMAS satisfying condition $C$}
From this section, we consider only the CMAS with finitely many $0$-factors. \\

In this section, we prove that all CMAS satisfying $C$ can have at most one non-trivial factor, and we do this in several steps.\\
\begin{proposition}
Let $(a_n)_{n \in \mathbf{N}}$ be a non-vanishing CMS taking values in the set $G=\left\{\zeta^r| r \in \mathbf{N} \right\}$, where $\zeta$ is a non-trivial $k$-th root of unity, having $u$ prime factors $a_{p_1}, a_{p_2}, ... a_{p_u}$; then, there exist $g \in G$ (where $a_{p_1}=g$) and a subword $\overline{w}_u$ appearing periodically in the sequence $(a_n)_{n \in \mathbf{N}}$ such that all its letters are different from $g$. Furthermore, the period does not have any other prime factor other than $p_1, p_2, ..., p_u$. What we mean by ``a word $\overline{w}_u$ appears periodically in the sequence $(a_n)_{n \in \mathbf{N}}$'' is that there exist two integers $m,l$ such that for all integers $n \in \mathbf{N}$ we have $\overline{a_{mn+l}a_{mn+l+1}...a_{mn+l+u-1}}=\overline{w}_u$, and we call $m$ the period.

\end{proposition}

\begin{proof}

The proof is by induction on  $u$. For $u=1$, the above statement is trivial. It is easy to check that the sequence $(a_{np_1^{k+1}+p_1^k})_{n \in \mathbf{N}}$ is all $1$'s, the period is $p_1^{k+1}$, and $g=a_{p_1}$. \\

Supposing that the statement is true for some $u=n_0$, let us consider the case $u=n_0+1$. We first consider the sequence $(a_n')_{n \in \mathbf{N}}$ defined as $a_n' = a_{\frac{n}{p_{n_0+1}^{v_{p_{n_0+1}}(n)}}}$, a sequence having $n_0$ prime factors, where $v_p(n)$ denotes the largest integer $r$ such that $p^r|n$. Using the hypothesis of induction, we obtain a subword $\overline{w}_{n_0}$ and two integers $m_{n_0},l_{n_0}$ satisfying the statement, that is to say, for all integers $n \in \mathbf{N}$, we have $\overline{a'_{m_{n_0}n+l_{n_0}}a'_{m_{n_0}n+l_{n_0}+1}...a'_{m_{n_0}n+l_{n_0}+n_0-1}}=\overline{w}_{n_0}$, furthermore, $m_{n_0}$  does not have any other prime factor other than $p_1, p_2, ..., p_{n_0}$. We can extract from the sequence $(a'_{m_{n_0}n+l_{n_0}})_{n \in \mathbf{N}}$ a sequence of the form $(a'_{m_{n_0'}n+l_{n_0}})_{n \in \mathbf{N}}$ such that $m_{n_0'}=m_{n_0}\prod_{j=1}^{n_0}p_j^{d_j}$ for some $d_j \in \mathbf{N}^{+}$ and $v_{p_j}(m_{n_0'}n+l_{n_0}+n_0)=v_{p_j}(l_{n_0}+n_0)$ for all $j \leq n_0$. In this case, the sequence $(a'_{m_{n_0'}n+l_{n_0}+n_0})_{n \in \mathbf{N}}$ is a constant sequence, say, all letters equal $C$.\\

Here, we consider two residue classes $N_1(n)$ and $N_2(n)$, separately satisfying the following conditions:

$$m_{n_0'}N_1(n) \equiv -l_{n_0}-n_0\mod p_{n_0+1}$$ $$m_{n_0'}N_1(n) \not\equiv -l_{n_0}-n_0\mod p_{n_0+1}^2$$

and $$m_{n_0'}N_2(n) \equiv -l_{n_0}-n_0\mod p_{n_0+1}^2$$ $$m_{n_0'}N_2(n) \not\equiv -l_{n_0}-n_0\mod p_{n_0+1}^3$$

In these two cases, we have $a_{m_{n_0^{'}}N_1(n)+l_{n_0}+n_0} = Ca_{p_{n_0+1}}$ and $a_{m_{n_0^{'}}N_2(n)+l_{n_0}+n_0} = Ca_{p_{n_0+1}}^2$ for all $n \in \mathbf{N}$. Because $a_{p_{n_0+1}} \neq 1$, there is at least one element of $ Ca_{p_{n_0+1}}, Ca_{p_{n_0+1}}^2$ not equal to $g$. If $N_i(n)$ is the associated residue class, then let us write down this residue class as $N_i(n)=p_{n_0+1}^{i+1}n+t$ for all integers $n$ with $t \in \mathbf{N}$, $i=1 \; \text{or} \;2$.\\

Now, let us choose $m_{n_0+1}=m_{n'_0}p_{n_0+1}^{i+1}$ and $ l_{n_0+1}=l_{n_0}+tm_{n'_0}$ so that the sequence $(a'_{m_{n_0+1}n+l_{n_0+1}})_{n \in \mathbf{N}}$ is a subsequence of $(a'_{m_{n_0}n+l_{n_0}})_{n \in \mathbf{N}}$; thus, from the hypothesis of induction, all subwords of $(a'_n)_{n \in \mathbf{N}}$ of length $n_0$ beginning at positions $m_{n_0+1}n+l_{n_0+1}$ for $n \in \mathbf{N}$ are the same, in other words, there exists a word $w_{n_0}$ such that 
$\overline{a'_{m_{n_0+1}n+l_{n_0+1}}a'_{m_{n_0+1}n+l_{n_0+1}+1}...a'_{m_{n_0+1}n+l_{n_0+1}+n_0-1}}=w_{n_0}$ for all $n \in \mathbf{N}$, and none of its letters equal $g$. Furthermore,  $a_{m_{n_0+1}n+l_{n_0+1}+n_0}=a_{m'_{n_0}N_i(n)+l_{n_0}+n_0}$ is constant and different from $g$ because of the choice of residue class. Now let us check that, for all $j$ such that $0 \leq j \leq n_0-1$, $p_{n_0+1} \nmid m_{n_0+1}n+l_{n_0+1}+j$. It is from the fact that $p_{n_0+1}| m_{n_0'}N_i(n)+l_{n_0}+n_0$ and $ m_{n_0'}N_i(n)+l_{n_0}+n_0-p_{n_0+1} =m_{n_0+1}n+l_{n_0+1}+n_0-p_{n_0+1}< m_{n_0+1}n+l_{n_0+1}+j $ for all $j$ such that $0 \leq j \leq n_0-1$, the last inequality is from the fact that $p_{n_0+1}>n_0+1$. Therefore, we conclude that, for all $n, j \in \mathbf{N}$ such that $0\leq j\leq n_0-1$, $v_{p_{n_0+1}}(m_{n_0+1}n+l_{n_0+1}+j)=0$. This means that all subwords of form $\overline{a_{m_{n_0+1}n+l_{n_0+1}}a_{m_{n_0+1}n+l_{n_0+1}+1}...a_{m_{n_0+1}n+l_{n_0+1}+n_0}}$ with $n \in \mathbf{N}$ are the same and of length $n_0+1$ and that none of its letters equals $g$; moreover, $m_{n_0+1}$ does not have any prime factor other than $p_1, p_2, ..., p_{n_0+1}$.
\end{proof}

\begin{proposition}
Let $(a_n)_{n \in \mathbf{N}}$ be a non-vanishing CMS defined on a finite set $G$ satisfying condition $\mathcal{C}$, and let $(a_n')_{n \in \mathbf{N}}$ be another CMS generated by the first $r$ prime factors of $(a_n)_{n \in \mathbf{N}}$, say, $a_{p_1}, a_{p_2} ,..., a_{p_r}$. If there is a subword $\overline{w}_r$ appearing periodically in $(a_n')_{n \in \mathbf{N}}$ and if the period does not have any prime factors other  than $p_1,p_2,...,p_r$, then this subword appears at least once in $(a_n)_{n \in \mathbf{N}}$.
\end{proposition}

\begin{proof}

Let us denote by $p_1, p_2...$ the sequence of prime numbers such that $a_{p_i} \neq 1$. Because of the hypothesis,  there are some integers $m_{r}, l_{r} \in \mathbf{N}$ such that  all subwords of the sequence $(a'_n)_{n \in \mathbf{N}}$, of length $r$ and beginning at positions $m_rn+l_r$ equal $\overline{w}_r$, for all $n \in \mathbf{N}$, furthermore, $m_r$ does not have any  prime factors other than $p_1,p_2,...,p_r$. Thus, the total number of occurrences of such subwords in the sequence $(a_{n})_{n \in \mathbf{N}}$ can be bounded by the inequality:
\begin{equation} 
\small{\#\left\{a_k| k \leq n, \overline{a_k,a_{k+1},...,a_{k+r-1}}=\overline{w}_{r}\right\} \geq\#\left\{a_k| k \leq n, k=m_rk'+l_r, k' \in \mathbf{N}; p_i \nmid k+j, \forall(i,j)\; \text{with} \; 0\leq j \leq r-1, i > r\right\}}.
\end{equation}
This inequality holds because at the right-hand side, we count only a part of occurrences of $\overline{w}_r$, namely those
  beginning at some position $m_rn+l_r$.\\
Let us consider the sequence defined as $N(t)=\prod_{j=1}^tp_{r+j}$; we have
\begin{equation} 
\begin{aligned} 
\#\left\{a_k| k \leq N(t)m_r+l_r, k=m_rk^{'}+l_r, k^{'} \in \mathbf{N}; p_i \nmid k+j,  \forall (i,j)\; \text{with} \;0\leq j \leq r-1, r < i \leq r +t \right\}=\prod_{j=1}^{t}(p_{r+j}-r) 
\end{aligned}
\end{equation}

This equality holds because of the Chinese reminder theorem and the fact that $p_{r+j} \nmid m_r$ and $p_{r+j} > r$ for all $j \geq 1$.

Therefore, we have
\begin{equation} 
\begin{aligned} 
&\#\left\{a_k| k \leq N(t)m_r+l_r, k=m_rk^{'}+l_r, k^{'} \in \mathbf{N}; p_i \nmid k+j,  \forall(i,j)\; \text{with} \; 0\leq j \leq r-1, i > r\right\}\\
>&\#\left\{a_k| k \leq N(t)m_r+l_r, k=m_rk^{'}+l_r, k^{'} \in \mathbf{N}; p_i \nmid k+j,  \forall(i,j)\; \text{with} \; 0\leq j \leq r-1, r < i \leq r +t \right\}\\
&-\#\left\{a_k| k \leq N(t)m_r+l_r, k=m_rk^{'}+l_r, k^{'} \in \mathbf{N}; p_i \mid k+j,  \forall(i,j)\; \text{with} \; 0\leq j \leq r-1, i > r +t \right\}\\
>&\#\left\{a_k| k \leq N(t)m_r+l_r, k=m_rk^{'}+l_r, k^{'} \in \mathbf{N}; p_i \nmid k+j,  \forall(i,j)\; \text{with} \; 0\leq j \leq r-1, r < i \leq r +t \right\}\\
&-\sum_{i > r +t}\#\left\{a_k| k \leq N(t)m_r+l_r, k=m_rk^{'}+l_r, k^{'} \in \mathbf{N}; p_i \mid k+j,  \forall j \; \text{with} \; 0\leq j \leq r-1\right\}\\
>&\prod_{j=1}^{t}(p_{r+j}-r) -r\sum_{i > r +t, p_i < N(t)+r}\lceil\frac{N(t)}{p_i}\rceil\\
>& \prod_{j=1}^{t}(p_{r+j}-r) -r\sum_{i > r +t,  p_i < N(t)+r}\frac{N(t)}{p_i}-r\pi(N(t)+r).
\end{aligned}
\end{equation}
where $\lceil a \rceil$ represents the smallest integer larger than $a$ and $\pi$ is the prime counting function.
However,
\begin{equation} 
\prod_{j=1}^{t}(p_{r+j}-r)=\prod_{j=1}^{t}\frac{p_{r+j}-r}{p_{r+j}}N(t) \geq \prod_{j=1}^{\infty}\frac{p_{r+j}-r}{p_{r+j}}N(t).
\end{equation}
The last formula can be approximated as $\prod_{j=1}^{\infty}\frac{p_{r+j}-r}{p_{r+j}}=\exp(\sum_{j=1}^{\infty}\log(\frac{p_{r+j}-r}{p_{r+j}}))=\exp(-\Theta(\sum_{j=1}^{\infty}\frac{r}{p_{r+j}}))$, and the last equality holds because $\log(1-x) \sim x $ when $x$ is small. Because of $C$, the above quantity does not diverge to $0$; we conclude that, if $t$ is sufficiently large, there exists a $c$ with $0<c<1$ such that $\prod_{j=1}^{t}(p_{r+j}-r) > cN(t)$.

On the other hand, we remark that for all $i > r +t$, $p_{i}^t>\prod_{j=1}^tp_{r+j}=N(t)$; thus, $p_{i}>N(t)^{\frac{1}{t}}$ and

\begin{equation} 
\sum_{i > r +t, p_i <N(t)+r}\frac{N(t)}{p_i} <N(t) \sum_{N(t)^{\frac{1}{t}}< p < N(t)+r} \frac{1}{p}.
\end{equation}

The term $N(t)^{\frac{1}{t}}$ can be bounded by
\begin{equation} 
N(t)^{\frac{1}{t}}=(\prod_{j=1}^{t}p_{r+j})^{\frac{1}{t}} \geq \frac{t}{\sum_{j=1}^{t}\frac{1}{p_{r+j}}} >  \frac{t}{\sum_{j=1}^{t}\frac{1}{q_{j}}}.
\end{equation}
where $q_j$ is the $j$-th prime number in $\mathbf{N}$. The first inequality is a consequence of the inequality of arithmetic and geometric means, which states that, for any $n$ positive numbers, say $a_1,a_2,...,a_n$, we have $\frac{a_1+a_2+...+a_n}{n} \geq (a_1a_2...a_n)^{\frac{1}{n}}$. For any $x \in \mathbf{N}$, $\#\left\{p_i| p_i \leq x\right\}\sim\frac{x}{\log(x)}$ and $\sum_{p_i \leq x} \frac{1}{p_i}\sim\log\log(x)$; thus, $N(t)^{\frac{1}{t}}$  tends to infinity when $t$ tends to infinity. Because of $C$, we can conclude that there exists some $t_0 \in \mathbf{N}$ such that, for all $t > t_0$, $\sum_{N(t)^{\frac{1}{t}}< p < N(t)+r} \frac{1}{p} < \frac{1}{2r}c$.\\

To conclude,  for all $t > t_0$,
\begin{equation} 
\begin{aligned} 
&\#\left\{a_k| k \leq N(t)m_r+l_r, k=m_rk'+l_r, k' \in \mathbf{N}; k+j \nmid p_i,  \forall(i,j)\; \text{with} \; 0\leq j \leq r-1,  \forall i > r\right\}\\
>&\prod_{j=1}^{t}(p_{r+j}-r) -r\sum_{k > r +t}\frac{N(t)}{p_k}-r\pi(N(t)+r)\\
>&cN(t)- \frac{1}{2}cN(t)-r\pi(N(t)+r).
\end{aligned}
\end{equation}
When $t$ tends to infinity, the set $\#\left\{a_k| k \leq n, \overline{a_k,a_{k+1},...,a_{k+r-1}}=\overline{w}_{r}\right\}$ is not empty.

\end{proof}

\begin{proposition}
Let $(a_n)_{n \in \mathbf{N}}$ be a $p$-CMAS, vanishing or not, satisfying condition $C$. Then, there exists at most one prime number $k$ such that $a_k \neq 1$ or $0$.
\end{proposition}

\begin{proof}
Suppose that the sequence $(a_n)_{n \in \mathbf{N}}$ has infinitely many prime factors not equal to $0$ or $1$. Let us consider first the sequence $(a'_n)_{n \in \mathbf{N}}$ defined as follows:
$$a'_n=a_{\frac{n}{\prod_{p_i \in \mathbf{Z}}p_i^{v_{p_i}(n)}}},$$
where $\mathbf{Z}=\left\{p| p \in \mathbf{P}, a_p=0\right\}$; because of Proposition 3, this set is finite.\\

Using Propositions 4 and 5, there exists a subword of length $p^{2q!}$, say, $\overline{v}_{p^{2q!}}$, appearing in $(a'_n)_{n \in \mathbf{N}}$ such that none of its letters equal $g=a'_{p_1}=a_{p_1}$, where $q$ is the number of states of the automaton generating $(a_n)_{n \in \mathbf{N}}$. Then, by construction, there is a subword of the same length, say, $\overline{w}_{p^{2q!}}$, appearing at the same position in the sequence $(a_n)_{n \in \mathbf{N}}$ such that none of its letters equal $g$. Extracting a subword $\overline{w'}_{p^{q!}}$ contained in $\overline{w}_{p^{2q!}}$ of the form $\overline{a_{up^{q!}}a_{up^{q!}+1}...a_{(u+1)p^{q!}-1}}$ for some $u \in \mathbf{N}$ and using Proposition 1, we have, for every $y$ such that $y \geq 1$ and every $m$ such that $0 \leq m \leq p^{yq!}-1$, $a_{up^{yq!}+m }\neq g$. In particular,
$$\lim_{y \to \infty}\frac{1}{p^{yq!}}\#\left\{a_s = g|up^{yq!} \leq s < (u+1)p^{yq!}-1 \right\} =0.$$
which contradicts the fact that $g$ has a non-zero natural density proved by Proposition 2.\\

Therefore, we have proven that the sequence $(a_n)_{n \in \mathbf{N}}$ must have finitely many prime factors. However, Corollary 2 of \cite{HU2017} proves that, in this case, the sequence $(a_n)_{n \in \mathbf{N}}$ can have at most one prime $k$ such that $a_k \neq 1$ or $0$.

\end{proof}
\section{Classification of CMAS}

In this section, we will prove that a CMAS is either strongly aperiodic or a Dirichlet-like sequence.

\begin{definition}
A sequence $(a_n)_{n \in \mathbf{N}}$ is said to be aperiodic if and only if, for any pair of integers $(s,r)$, we have
$$\lim_{N \to \infty}\frac{\sum_{i=0}^{N}a_{si+r}}{N}=0.$$
\end{definition}

\begin{definition}

Let $\mathcal{M}$ be the set of completely multiplicative functions. Let $\mathbf{D}: \mathcal{M} \times \mathcal{M} \times \mathbf{N} \to [0, \infty]$ be given by

$$\mathbf{D}(f,g,N)^2=\sum_{p \in \mathbf{P}\cap [N] }\frac{1-Re(f(p)\overline{g(p)})}{p}$$
and $M: \mathcal{M} \times \mathbf{N} \to [0, \infty)$ be given by

$$M(f, \mathbf{N})=\min_{|t|\leq N}\mathbf{D}(f,n^{it},N)^2$$
A sequence $(a_n)_{n \in \mathbf{N}}$ is said to be strongly aperiodic if and only if $M(f\chi,N) \to \infty$ as $N \to \infty$ for every Dirichlet character $\chi$.
\end{definition}

\begin{definition}
A sequence $(a_n)_{n \in \mathbf{N}}$ is said to be  (trivial) Dirichlet-like if and only if there exists a (trivial) Dirichlet character $X(n)_{n \in \mathbf{N}}$ such that there exists at most one prime number $p$ satisfying $a_p \neq X(p)$.
\end{definition}

\begin{proposition}
Let $(a_n)_{n \in \mathbf{N}}$ be a CMAS; then, either there exists a Dirichlet character $(X(n))_{n \in \mathbf{N}}$ such that the sequence $(a_nX(n))_{n \in \mathbf{N}}$ is a trivial Dirichlet-like character or it is strongly aperiodic.
\end{proposition}

\begin{proof}
First, it is easy to check that there is an integer $r$ such that $a_p$ is the $r$-th root of unity for all but finitely many primes $p$ (see Lemma 1 \cite{Puchta2011}). If $(a_n)_{n \in \mathbf{N}}$ is not strongly
aperiodic, then because of Proposition 6.1 in \cite{Frant2018}, there exists a Dirichlet character $(X(n))_{n \in \mathbf{N}}$ such that $$\lim_{N \to \infty}\mathbf{D}(a,X,N) < \infty (*).$$ However, the sequence $(a_n\overline{X(n)})_{n \in \mathbf{N}}$ is also CMAS and satisfies condition $\mathcal{C}$; the last fact is from $(*)$. Because of Proposition 6, $(a_n\overline{X(n)})_{n \in \mathbf{N}}$ is a trivial Dirichlet-like character.
\end{proof}

\begin{proposition}
Let $(a_n)_{n \in \mathbf{N}}$ be a CMAS and $X_t(n)_{n \in \mathbf{N}}$ be a Dirichlet character (mod $t$). If the sequence $(a_nX_t(n))_{n \in \mathbf{N}}$ is the trivial Dirichlet-like character (mod $t$), then $(a_n)_{n \in \mathbf{N}}$ is either a Dirichlet character (mod $t$) or a Dirichlet-like character $a_n=\epsilon^{v_p(n)}X(\frac{n}{p^{v_p(n)}})$, where $p$ is a prime divisor of $t$ and $\epsilon$ is a root of unity.
\end{proposition}

\begin{proof}
Let $(a_n)_{n \in \mathbf{N}}$ be a CMAS satisfying the above hypothesis; then, all possibilities for such $(a_n)_{n \in \mathbf{N}}$ are the sequences of the form
$$a_n=\prod_{i=1}^m\epsilon_i^{v_{p_i}(n)}X\left(\frac{n}{\prod_{i=1}^mp_i^{v_{p_i}(n)}}\right),$$
for each $n$, where $\epsilon_i$ are all non-zero complex numbers and $p_i$ are all prime factors of $t$.\\

Let us consider the Dirichlet sequence $f(s)$ associated with the sequence $(a_n)_{n \in \mathbf{N}}$, which can be written as
$$f(s)=L(s,X_t)\prod_{i=1}^m\frac{1-\frac{1}{p_i^s}}{1-\frac{a_{p_i}}{p_i^s}}.$$
Therefore, all the poles of $f(s)$ can be found on
$$s=\frac{\log a_{p_i}+2im\pi}{\log p_i},$$
for all $i$ such that $1 \leq i \leq m$ and $n \in \mathbf{N}$.\\

However, if $(a_n)_{n \in \mathbf{N}}$ is a $k$-automatic sequence for some integer $k$, then the poles should be located at points
$$s = \frac{\log \lambda}{\log k} + \frac{2im\pi}{\log k} -l +1,$$
where $\lambda$ is any eigenvalue of a certain matrix defined from the sequence $(\chi_n)_{n \in \mathbf{N}}$, and $m \in \mathbf{Z}, l \in \mathbf{N}$, and $\log$ is a branch of the complex logarithm \cite{ALLOUCHE2000}. By comparing the two sets of possible locations of poles for the same function, we can see that there is at most one $a_{p_i} \neq 0$.
\end{proof}

\section {Conclusion}

In this section, we conclude this article by proving that strongly aperiodic CMAS do not exist. To do so, we define the block complexity of sequences.

\begin{definition}

Let $(a_n)_{n \in \mathbf{N}}$ be a sequence. The block complexity of $(a_n)_{n \in \mathbf{N}}$ is a sequence, which will be denoted by $(p(k))_{k \in \mathbf{N}}$, such that $p(k)$ is the number of subwords of length $k$ that occur (as consecutive values) in $(a_n)_{n \in \mathbf{N}}$
\end{definition}

\begin{proposition}
If $(a_n)_{n \in \mathbf{N}}$ is a CMAS, then it is not strongly aperiodic.
\end{proposition}

\begin{proof}

From Theorem 2 in (\cite{Bernard2018}) and the following remark, the block complexity of the sequence $(a_n)_{n \in \mathbf{N}}$ should satisfy the property that $\lim_{n \to \infty}\frac{p(n)}{n} = \infty$, which contradicts the fact that the block complexity of an automatic sequence is bounded by a linear function  \cite{cob}. Therefore, the non-existence of strongly aperiodic CMAS is proved.
\end{proof}

\begin{theorem}
Let $(a_n)_{n \in \mathbf{N}}$ be a CMAS; then, it can be written in the following form:\\
-either there is at most one prime $p$ such that $a_p \neq 0$ and $a_q=0$ for all other primes $q$\\
-or $a_n=\epsilon^{v_p(n)}X(\frac{n}{p^{v_p(n)}})$, where $(X(n))_{n \in \mathbf{N}}$ is a Dirichlet character.\\
\end{theorem}

\section{Acknowledgement}

We found some results in the recent literature on similar topics that have applications to the classification of CMAS. In \cite{lem}, the authors proved that all continuous observables in a substitutional dynamical system $(X_{\theta},S)$ are orthogonal to any
bounded, aperiodic, multiplicative function, where $\theta$ represents a primitive uniform substitution and $S$ is the shift operator. As an application, all multiplicative and automatic sequences produced by primitive automata are Weyl rationally almost periodic. We remark that a sequence $(b_n)_{n \in \mathbf{N}}$ is called Weyl rationally almost periodic if it can be approximated by periodic sequences over same alphabet in the pseudo-metric
$$d_W(a,b)= \limsup_{N \to \infty}\sup_{l \geq 1}\frac{1}{N}|\left\{l \leq n <l+N: a(n) \neq b(n)\right\}|.$$ 

In \cite{klurman2018}, the authors considered general multiplicative functions with the condition \\
$\liminf_{N \to \infty} |b_{n+1}-b_{n}| >0$. They proved that if $(b_n)_{n \in \mathbf{N}}$ is a completely multiplicative sequence, then most primes, at a fixed
power, give  the same values as a Dirichlet character.

\bibliographystyle{alpha}
\bibliography{citations}

\begin{thebibliography}{AFP00}

\bibitem[AFP00]{ALLOUCHE2000}
J.-P. Allouche, M.~Mendès France, and J.~Peyrière.
\newblock {Automatic Dirichlet series}.
\newblock {\em Journal of Number Theory}, 81(2):359 -- 373, 2000.

\bibitem[AG18]{ALLOUCHE2018}
J.-P. Allouche and L.~Goldmakher.
\newblock {Mock characters and the Kronecker symbol}.
\newblock {\em Journal of Number Theory}, 192:356 -- 372, 2018.

\bibitem[Cob72]{cob}
A.~Cobham.
\newblock {Uniform tag sequences}.
\newblock {\em Mathematical systems theory}, 6(1):164--192, Mar 1972.

\bibitem[FH19]{Bernard2018}
N.~Frantzikinakis and B.~Host.
\newblock {Furstenberg Systems of Bounded Multiplicative Functions and
  Applications}.
\newblock {\em {International Mathematics Research Notices}, to appear}, 2019.
\newblock preprint, \url{http://arxiv.org/abs/1804.18556}.

\bibitem[Fra18]{Frant2018}
N.~Frantzikinakis.
\newblock {An Averaged Chowla and Elliott Conjecture Along Independent
  Polynomials}.
\newblock {\em International Mathematics Research Notices},
  2018(12):3721--3743, 2018.

\bibitem[Hu17]{HU2017}
Y.~Hu.
\newblock {Subword complexity and non-automaticity of certain completely
  multiplicative functions}.
\newblock {\em Advances in Applied Mathematics}, 84:73 -- 81, 2017.

\bibitem[KM17]{klurman2018}
O.~Klurman and A.~P. Mangerel.
\newblock Rigidity theorems for multiplicative functions.
\newblock {\em Mathematische Annalen}, 07 2017.

\bibitem[LM18]{lem}
M.~Lemańczyk and C.~Müllner.
\newblock Automatic sequences are orthogonal to aperiodic multiplicative
  functions.
\newblock 11 2018.
\newblock preprint, \url{http://arxiv.org/abs/1811.00594}.

\bibitem[Ruz77]{Ruzsa1977}
I.~Z. Ruzsa.
\newblock {General multiplicative functions}.
\newblock {\em Acta Arithmetica}, 32(4):313--347, 1977.

\bibitem[SP11]{Puchta2011}
J.-C. Schlage-Puchta.
\newblock {Completely multiplicative automatic functions}.
\newblock {\em Integers}, 11:8, 2011.

\end{thebibliography}

\end{document}